\documentclass{amsart}

\usepackage{amsmath,amsfonts,amsthm,tikz}
\usetikzlibrary{arrows}
\usepackage[margin=1in]{geometry}

\DeclareMathOperator{\Assh}{Assh}
\DeclareMathOperator{\Ht}{ht}
\newcommand{\lengthR}[2]{\lambda_{#1}\left(#2\right)}
\newcommand{\m}{\mathfrak{m}}
\DeclareMathOperator{\Min}{Min}
\newcommand{\N}{\mathbb{N}}
\newcommand{\scl}[2]{{#2}^{\mathrm{cl}_{#1}}}
\newcommand{\p}{\mathfrak{p}}
\newcommand{\Q}{\mathbb{Q}}

\newcommand{\wsc}[2]{{#2}^{\left\{#1\right\}}}
\newcommand{\Z}{\mathbb{Z}}

\newtheorem{theorem}{Theorem}[section]
\newtheorem*{theorem*}{Theorem}
\newtheorem*{lemma*}{Lemma}
\newtheorem{lemma}[theorem]{Lemma}
\newtheorem{corollary}[theorem]{Corollary}
\theoremstyle{definition}
\newtheorem{definition}[theorem]{Definition}
\newtheorem{example}[theorem]{Example}

\begin{document}
\title{Fundamental Results on $s$-Closures}
\author{William D.\ Taylor}
\address{Tennessee State University, Nashville, Tennessee}
\email{wtaylo17@tnstate.edu}

\begin{abstract}  This paper establishes the fundamental properties of the $s$-closures, a recently introduced family of closure operations on ideals of rings of positive characteristic.  The behavior of the $s$-closure of homogeneous ideals in graded rings is studied, and criteria are given for when the $s$-closure of an ideal can be described exactly in terms of its tight closure and rational powers.  Sufficient conditions are established for the weak $s$-closure to equal to the $s$-closure.  A generalization of the Brian\c con-Skoda theorem is given which compares any two different $s$-closures applied to powers of the same ideal.
\end{abstract}

\maketitle

\section{Introduction}

In \cite{Taylor-Interpolating}, the author introduced a family of closure operations on the ideals of noetherian rings of positive characteristic which lie between and interpolate between integral closure and tight closure of those ideals.  For a real number $s\geq 1$, the \emph{weak $s$-closure} of an ideal $I$ in a ring $R$ is the set of $x\in R$ such that there exists $c\in R$, not in any minimal prime, such that $cx^q\in I^{\lceil sq\rceil}+I^{[q]}$ for all sufficiently high powers $q$ of the characteristic of $R$.  We denote the weak $s$-closure of $I$ by $\wsc{s}{I}$.  The \emph{$s$-closure} of $I$ is the ideal obtained by applying the weak $s$-closure repeatedly until the ideal stabilizes.

The $s$-closures are related to the $s$-multiplicity function, which similarly interpolates between the Hilbert-Samuel and Hilbert-Kunz multiplicities of an ideal.  The $s$-multiplicity of an $\m$-primary ideal $I$ in a local ring $(R,\m)$ is
\[e_s(I)= \lim_{q\to\infty}\frac{\lambda\left(R/(I^{\lceil sq\rceil}+I^{[q]})\right)}{q^d\mathcal{H}_s(d)},\]
where $\mathcal{H}_s(d)$ is a normalizing factor depending only on $s$ and the Krull dimension $d$ of $R$.

 The strongest result on the subject of $s$-closures in \cite{Taylor-Interpolating} is Theorem 4.6, which states that if $I$ and $J$ are $\m$-primary ideals of a ring $R$ and $\wsc{s}{I}=\wsc{s}{J}$, then $e_s(I)=e_s(J)$.  The same theorem gives a partial converse: if $R$ is an $F$-finite complete domain, and $I\subseteq J$, and $e_s(I)=e_s(J)$, then $\wsc{s}{I}=\wsc{s}{J}$.  Furthermore, in this case the weak $s$-closure is the $s$-closure, i.e.\  $\wsc{s}{I}=\scl{s}{I}$.  In this paper, we show in Theorem~\ref{theorem - MP to SM} that the domain hypothesis for the converse direction may be weakened to an unmixed hypothesis.

 This paper's purpose is to develop significantly more of the theory of $s$-closures, particularly to establish the results that will be essential to further study. The three main goals of this paper are to understand the structure of the $s$-closure in the graded case, identify situations in which $\wsc{s}{I}=\scl{s}{I}$, and to compare the $s$-closures for different values of $s$ using a generalization of the Brian\c con-Skoda theorem.

 Here we record those results in the paper we believe will be most relevant to future work.  In some cases the statement of the full theorem is slightly stronger but more technical.

\begin{lemma*}[\ref{lem - t-closure is s-closed}] If $R$ is a ring of characteristic $p>0$, $I\subseteq R$ is an ideal, and $1\leq t<s$, then $\wsc{s}{\left(\wsc{t}{I}\right)}=\wsc{t}{I}$.
\end{lemma*}

\begin{theorem*}[\ref{thm - all ideals have MP}] Let $R$ be a ring of characteristic $p>0$, $I\subseteq R$ an ideal, and $s\geq1$.  For any $x\in R$, $x\in \wsc{s}{I}$ if and only if $\overline{x}\in \wsc{s}{(IR/\p)}$ for all $\p\in \Min R$. 
\end{theorem*}

 \begin{theorem*}[\ref{thm - rational powers to sclosure}, \ref{thm - sclosure is mixed power criterion}]
Let $R$ be a ring of characteristic $p>0$, $I\subseteq R$ an ideal, and $s\geq 1$ a rational number.  We have that $I^*+I_s\subseteq \wsc{s}{I}$, where $I_s$ is the $s$th rational power of $I$.  Furthermore, equality holds if $I$ is a monomial ideal in a polynomial or semigroup ring over a field.
 \end{theorem*}

\begin{theorem*}[\ref{thm-graded sclosure}, \ref{thm - min degree}, \ref{thm - elements forced in}]
Let $R$ be an $\N$-graded ring of characteristic $p>0$, $I\subseteq R$ a homogeneous ideal, $x\in R$ a homogeneous element, and $s\geq 1$.
\begin{enumerate}
	\item $\wsc{s}{I}$ and $\scl{s}{I}$ are homogeneous ideals.
	\item If all generators of $I$ have degree at least $d$ and $x\in\wsc{s}{I}\setminus I^*$, then $\deg x\geq sd$.
	\item If $(R,\m)$ is graded local, $I$ is $\m$-primary and generated in degree at most $d$, and $\deg x\geq sd$, then $x\in \wsc{s}{I}$.
\end{enumerate}
\end{theorem*}

\begin{theorem*}[\ref{cor - ideals with LS}] For the following classes of ideals, $\wsc{s}{I}=\scl{s}{I}$.
\begin{enumerate}
	\item Monomial ideals in polynomial rings, or more generally affine semigroup rings, over a field
	\item Principal ideals
	\item Powers of $R_+$, where $R$ is an $\N$-graded ring generated in degree 1 over $R_0$ and $R_+$ is generated by all elements of positive degree.
\end{enumerate}
\end{theorem*}

\begin{theorem*}[\ref{theorem - Briancon-Skoda}] 
Let $R$ be a ring, $1\leq t< s$, and $I$ an ideal of $R$.  If $r\geq \frac{(\mu(I)-1)(s-t)}{t(s-1)}$, then for all $n\in\N$, $\wsc{t}{\left(I^{n+r}\right)}\subseteq \wsc{s}{\left(I^n\right)}$.
\end{theorem*}

An outline of this paper is as follows.  In section~\ref{sec - preliminaries}, we give the basic definitions and results on powers of ideals that we use throughout the paper.  We also record some results about rational powers of ideals, and prove a characterization of them which is particularly relevant to us.  In section~\ref{sec - graded rings}, we consider the $s$-closure of homogeneous ideals in graded rings.  We obtain degree conditions which can be used in some cases to check the membership or non-membership of a homogeneous element in the $s$-closure of a homogeneous ideal.  Section~\ref{sec - ideal conditions} considers the question of when $\wsc{s}{I}=\scl{s}{I}$, and gives some sufficient conditions on $I$ for equality to hold.  Section~\ref{sec - BS} includes our generalization of the Brian\c con-Skoda theorem which compares any two $s$-closures.

\section{Preliminaries}\label{sec - preliminaries}

Throughout this paper, all rings $R$ are assumed to be commutative and noetherian, and the 
notation $R^\circ$ indicates the set of all elements of $R$ not in any minimal prime ideal.  For an ideal $I$, we use $\mu(I)$ for the minimal number of generators of $I$.  When we work with a ring of positive characteristic $p$, the symbols $q$ and $q'$ stand for positive integer powers of $p$.  For an ideal $I$ in a ring of characteristic $p>0$, the ideal $I^{[q]}=(f^q\mid f\in I)$ is called the $q$th Frobenius power of $I$, and is generated as an ideal by the $q$th powers of any set of generators of $I$. 

We are interested in the relationships between ordinary and Frobenius powers of ideals.  In particular, we rely on the following result.

\begin{lemma}\label{lem - basic interaction} If $R$ is a ring of characteristic $p>0$, $h\geq 0$ is a real number, $I$ is an ideal of $R$, and $q$ is a power of $p$, then $I^{\lceil h\rceil}\subseteq \left(I^{[q]}\right)^{\lfloor h/q -\mu(I)+1\rfloor}\subseteq \left(I^{[q]}\right)^{\lceil h/q -\mu(I)\rceil}$ 
\end{lemma}

\begin{proof} Let $x_1,\ldots, x_{\mu(I)}$ be a set of generators for $I$.  For any generator $x$ of $I^{\lceil h\rceil}$, there exist $a_i,b_i\in\N$ such that $b_i<q$, $\sum a_iq+b_i=\lceil h\rceil$, and 
\[x=\prod_{i=1}^{\mu(I)} x_i^{qa_i+b_i}=\prod_{i=1}^{\mu(I)}\left( x_i^{q}\right)^{a_i}\cdot\prod_{i=1}^{\mu(I)} x_i^{b_i}\in \left(I^{[q]}\right)^{\sum_i a_i}.\]
 Furthermore,
\[\sum_{i=1}^{\mu(I)} a_i = \sum_{i=1}^{\mu(I)} \frac{qa_i+b_i-b_i}{q}=\frac{\lceil h\rceil}{q}-\sum_i \frac{b_i}{q}\geq \frac{\lceil h\rceil}{q}-\mu(I)\frac{q-1}{q}>\frac{ h}{q}-\mu(I)\]
Therefore, since $\sum_i a_i$ is an integer,
$\sum_i a_i\geq   \left\lfloor h/q -\mu(I)+1\right\rfloor$.

The last containment is implied by the fact that $\lceil \alpha\rceil\leq \lfloor \alpha+1\rfloor$ for all real $\alpha$. 
\end{proof}

\subsection*{Mixed Powers and $s$-Closure}

 Given a ring $R$ and ideal $I$, the integral closure $\overline{I}$ of $I$ is the set of all $x\in R$ such that there exists $c\in R^\circ$ such that $cx^n\in I^n$ for infinitely many positive integers $n$, or equivalently all sufficiently large integers $n$ \cite[Corollary 6.8.12]{Huneke-Swanson-IntegralClosure}.  When $R$ has characteristic $p>0$, the tight closure $I^*$ of $I$ is the set of all $x\in R$ such that there exists $c\in R^\circ$ such that $cx^q\in I^{[q]}$ for all $q\gg 0$.  The similarity between these two descriptions suggests a method of interpolating between the two closures.    We begin by considering a set of ideals which interpolate between ordinary powers and Frobenius powers of an ideal.
 
 \begin{definition} Let $R$ be a ring of characteristic $p>0$, $s\geq 1$ a real number, and $I$ an ideal of $R$.  For any $q$, the \emph{$(s,q)$ mixed power of $I$} is 
\[I^{(s,q)} = I^{\lceil sq\rceil}+I^{[q]}.\]
\end{definition}

Note that $I^{(1,q)}=I^{q}$, and that if $s\geq\mu(I)$, then $I^{(s,q)}=I^{[q]}$.  Furthermore, we have that if $s>t$, then $I^{(s,q)}\subseteq I^{(t,q)}$.  Therefore, the ideals $I^{(s,q)}$ form a decreasing family of ideals parameterized by $s$.  
In \cite{Taylor-Interpolating}, the author used the mixed powers defined above to construct a family of closures which lie between integral closure and tight closure.

\begin{definition}\cite[Definition 4.1]{Taylor-Interpolating} \label{def - s-closure} Let $R$ be a ring of characteristic $p>0$, $s\geq 1$ a real number, and $I$ an ideal of $R$.  The \emph{weak $s$-closure} of $I$, denoted $\wsc{s}{I}$, is the set of all $x\in R$ such that there exists $c\in R$ such that for all $q\gg 0$, $cx^q\in I^{(s,q)}$.
\end{definition}

It is easy to see that $\wsc{s}{I}$ is an ideal containing $I$, and that if $I\subseteq J$ then $\wsc{s}{I}\subseteq \wsc{s}{J}$, but it is not clear that the weak $s$-closure is idempotent.  Thus, to construct a true closure operation, we apply the weak $s$-closure repeatedly.

\begin{definition}\cite[Definition 4.3]{Taylor-Interpolating}  Let $R$ be a ring of characteristic $p>0$, $s\geq 1$, and $I$ an ideal of $R$.  The \emph{$s$-closure} of $I$, denoted $\scl{s}{I}$, is the ideal at which the following increasing chain of ideals stabilizes:
\[I\subseteq\wsc{s}{I}\subseteq\wsc{s}{\left(\wsc{s}{I}\right)}\subseteq\wsc{s}{\left(\wsc{s}{\left(\wsc{s}{I}\right)}\right)} \cdots.\]
\end{definition}

It is not known whether $\wsc{s}{I}=\scl{s}{I}$ for all $s$ and ideals $I$.  The condition that $\wsc{s}{I}=\scl{s}{I}$ is explored in Section~\ref{sec - ideal conditions}.

Since $s>t$ implies $I^{(s,q)}\subseteq I^{(t,q)}$, we have that if $s>t$, then  $\wsc{s}{I}\subseteq \wsc{t}{I}$ and $\scl{s}{I}\subseteq \scl{t}{I}$.  Moreover, since $I^{[q]}\subseteq I^{(s,q)}\subseteq I^q$ for all ideals $I$, $s\geq 1$ and $q$, we have that $I^*\subseteq \wsc{s}{I}\subseteq \scl{s}{I}\subseteq \overline{I}$ for all $s$ and $I$.

 Furthermore, when $s$ is very small or very large, some of the containments above become equalities.

\begin{theorem}  If $R$ is a ring of characteristic $p>0$ and $I$ an ideal of $R$, then the following hold.
\begin{enumerate}
\item \label{s=1 is integral}$\wsc{1}{I}=\scl{1}{I}=\overline{I}$.

\item \label{s big is tight}If either $s\geq \mu(I)$ or $s>\mu(J)$, where $J$ is a reduction of $I$, then $\wsc{s}{I}=I^*$.  In particular, if $R$ is local with infinite residue field and $s>\dim R$, then $\wsc{s}{I}=\scl{s}{I}=I^*$.
\end{enumerate}
\end{theorem}

\begin{proof} \eqref{s=1 is integral} If $x\in \wsc{1}{I}$, then there exists $c\in R^\circ$ such that $cx^q\in I^{(1,q)}=I^q$ for all $q\gg 0$, and hence $x\in \overline{I}$.  If $x\in \overline{I}$, then there exists $c\in R^\circ$ such that $cx^n\in I^n$ for all $n\gg 0$, and hence $cx^q\in I^q=I^{(1,q)}$ for all $q\gg 0$, and so $x\in \wsc{1}{I}$.  Therefore $\wsc{1}{I}=\overline{I}$, and since this holds for all ideals $I$, we have that weak $1$-closure and tight closure are the same operation, hence weak $1$-closure is idempotent, i.e.\ $\wsc{1}{I}=\scl{1}{I}$.

\eqref{s big is tight}  If $s\geq \mu(I)$ and $x\in \wsc{s}{I}$, then there exists $c\in R^\circ$ such that $cx^q\in I^{(s,q)}=I^{[q]}$ for all $q\gg0$, and therefore $x\in I^*$.  

Suppose $s>\mu(J)$, where $J$ is a reduction of $I$ with reduction number $w$ and $s>\mu(J)$.  Since $I^*\subseteq \wsc{s}{I}$ we need only show that $\wsc{s}{I}\subseteq I^*$.  Let $x\in \wsc{s}{I}$ and let $c\in R^\circ$ such that $cx^q\in I^{(s,q)}$ for $q\gg 0$.  Now, for large enough $q$, we have that $\frac{w}{q}\leq s-\mu(J)$, and so  by Lemma~\ref{lem - basic interaction},
\[cx^q\in I^{(s,q)} = I^{\lceil sq\rceil}+ I^{[q]} \subseteq J^{\lceil sq-w\rceil}+I^{[q]}
\subseteq \left(J^{[q]}\right)^{\lfloor s-w/q -\mu(J)+1\rfloor} + I^{[q]}\subseteq J^{[q]}+I^{[q]} = I^{[q]}.\]
Therefore $x\in I^*$.

If $R$ is local with infinite residue field, then every ideal has a reduction generated by at most $\dim R$ elements.  Hence in this case, if $s>\dim R$, then weak $s$-closure and tight closure are the same operation, and so in particular weak $s$-closure is idempotent.  Hence for any ideal $I$, $\scl{s}{I}=\wsc{s}{I}=I^*$. 
\end{proof}

\subsection*{Rational Powers}

The notion of rational powers of ideals is related to that of $s$-closure.  In particular, rational powers will be used to describe the $s$-closures of certain kinds of ideals in graded rings in Theorem~\ref{thm - sclosure is mixed power criterion}.  The presentation here is based on \cite[Section 10.5]{Huneke-Swanson-IntegralClosure}.

\begin{definition} Let $R$ be a ring, $I\subseteq R$ an ideal, and $\alpha\in \Q_{\geq 0}$.  The \emph{$\alpha$th rational power of $I$}, denoted $I_\alpha$, is the set of all $x\in R$ such that $x^b\in \overline{I^a}$, where $\alpha=\frac{a}{b} $, $a,b\in\N$.
\end{definition}

The ideal $I_\alpha$ does not depend on the choice of representation of $\alpha$ as a fraction.  
The most important property of rational powers that we will use is the following.

\begin{theorem}\label{thm - discrete rational powers}(\cite[Propositon 10.5.5]{Huneke-Swanson-IntegralClosure}) Let $R$ be a ring and $I\subseteq R$ an ideal of positive height.  There exists $e\in\N$ such that for all $\alpha\in\Q_{\geq 0}$, $I_\alpha = I_{\lceil \alpha e\rceil/e}$.
\end{theorem}

We can use Theorem~\ref{thm - discrete rational powers} to give an alternate description of $I_\alpha$ which simultaneously relates to the $s$-closure and doesn't depend on the representation of $\alpha$.

\begin{lemma}\label{lem - rational power description} Let $R$ be a ring of any characteristic, $I$ an ideal of positive height, and $\alpha\in \Q_{\geq 0}$.  \begin{enumerate}
\item \label{rational to ceiling}If $x\in I_\alpha$, then there exists $c\in R^\circ$ such that for all $n\gg\N$, $cx^n\in I^{\lceil \alpha n\rceil}$.
\item \label{ceiling to rational}If there exists $c\in R^\circ$ and $m\in \N_{>0}$ such that for infinitely many $n$, $cx^{m^n}\in I^{\lceil \alpha m^n\rceil}$, then $x\in I_\alpha$.
\end{enumerate}
\end{lemma} 

\begin{proof} \eqref{rational to ceiling} Suppose $\alpha = \frac{a}{b}$ with $a,b\in \N$ and $x\in I_\alpha$, so that $x^b\in \overline{I^a}$.  Therefore there exists $c'\in R^\circ$ such that for all $k\gg 0$, $c'x^{bk}\in I^{ak}$.  Let $c''\in I^a\cap R^\circ$ and $n\gg 0$.  We have that
\[c'c''x^{b\lfloor n/b\rfloor}\in c''I^{a\lfloor n/b\rfloor}\subseteq I^{a\lfloor n/b\rfloor +a}\subseteq I^{\lceil na/b\rceil} = I^{\lceil n\alpha\rceil}\]  
So, setting $c=c'c''$ and noting that $x^n\in (x^{b\lfloor n/b\rfloor})$, we're done.

\eqref{ceiling to rational} Suppose that there exists $c\in R^\circ$ and $m\in \N_{>0}$ such that for infinitely many $n$, $cx^{m^n}\in I^{\lceil \alpha m^n\rceil}$.  Let $e\in \N$ such that for any $\beta\in \Q_{\geq 0}$, $I_\beta = I_{\lceil \beta e\rceil /e}$.  Choose $a,k\in \N$ such that $\frac{a}{m^k}<\alpha$ and $\left\lceil \frac{a}{m^k}e\right\rceil = \lceil \alpha e\rceil$.  Now, for infinitely many $n\geq k$, we have that
\[c\left(x^{m^k}\right)^{m^{n-k}} = cx^{m^n}\in I^{\lceil \alpha m^n\rceil} \subseteq I^{\lceil (a/m^k)m^n\rceil} = I^{am^{n-k}}.\]
Therefore, $x^{m^k}\in \overline{I^a}$, and so $x\in I_{a/m^k}=I_\alpha$.
\end{proof}

This description of the rational powers gives us another bound for the $s$-closure of an ideal.

\begin{theorem}\label{thm - rational powers to sclosure} Let $R$ be a ring of characteristic $p>0$ and $I\subseteq R$ an ideal of positive height.  If $s\geq 1$ is a rational number then $I^*+I_s\subseteq \wsc{s}{I}$.
\end{theorem}

\begin{proof} That $I^*\subseteq \wsc{s}{I}$ is already known.  Suppose $x\in I_s$.  By Lemma~\ref{lem - rational power description}, there exists $c\in R^\circ$ such that for all $q$, $cx^q\in I^{\lceil sq\rceil}\subseteq I^{(s,q)}$.  Therefore $x\in \wsc{s}{I}$.
\end{proof}

\section{Graded Rings}\label{sec - graded rings}

Here we establish the essential facts about the $s$-closure of homogeneous ideals in graded rings.  Throughout this section, for a semigroup $G$, a $G$-graded ring $R=\bigoplus_{g\in G} R_g$, and $x\in R$,  we write $x_g$ for the homogeneous component of $x$ lying in $R_g$.  If $I\subseteq R$ is an ideal, we write $[I]_g$ for $I\cap R_g$.  If $g=(g_1,g_2,\ldots, g_n)\in \Z^n$, then we write $\|g\|_\infty=\max\{|g_i|\mid i=1,\ldots, n\}$.
We begin with an expected result.

 \begin{theorem}  \label{thm-graded sclosure} If $R$ is a $\Z^n$-graded ring of characteristic $p>0$, $I$ is a homogeneous ideal of $R$, and $s\geq 1$, then $\wsc{s}{I}$ and $\scl{s}{I}$ are homogeneous ideals.  Furthermore, if $R$ is a domain and $x\in \wsc{s}{I}$, there exists a homogeneous element $c$ such that $cx^q\in I^{(s,q)}$ for all $q\gg 0$.
\end{theorem}

\begin{proof} Let $x=\sum_{j\in \Z^n}x_j \in \wsc{s}{I}$ and $c=\sum_{i\in \Z^n} c_i \in R^\circ$ such that $cx^q\in I^{(s,q)}$ for all $q\gg 0$.  
Let $i^*=\max\{\|i-i'\|_\infty\mid c_i,c_{i'}\neq 0\}$.
If  $c_ix_j^q\neq 0$ and $c_{i'}x_{j'}^q\neq 0$ have the same degree, then $i+qj=i'+qj'$, and so $i-i'=q(j'-j)$.  If in addition $q>i^*$, we must have that $j=j'$ and $i=i'$.  Therefore each  nonzero homogeneous component of $cx^q$ is $c_ix_j^q$ for some $i,j$.  Since $I$ is homogeneous, so is  $I^{(s,q)}$, and therefore for $q\gg 0$, we have that $c_ix_j^q\in I^{(s,q)}$.  Hence, for each $j$, $cx_j^q\in I^{(s,q)}$, which shows that each $x_j\in \wsc{s}{I}$. If $R$ is a domain, then for each nonzero $c_i$, $c_ix^q\in  I^{(s,q)}$ for $q\gg 0$, which shows the last statement.

Since $\wsc{s}{I}$ is homogeneous, so is $\wsc{s}{\left(\wsc{s}{I}\right)}$, and each time we take the weak $s$-closure we preserve homogeneity.  After a finite number of steps we will reach $\scl{s}{I}$, which shows that $\scl{s}{I}$ is homogeneous.
\end{proof}

Our primary goal will be to establish necessary and sufficient degree conditions for a homogeneous ring element to belong to the $s$-closure of an ideal based on the degrees of its generators.

 \begin{theorem} \label{thm - min degree} Let $I$ be a homogeneous ideal in an $\N^n$-graded ring $R$.   If  $x\in \wsc{s}{I}\setminus I^*$ is a homogeneous element, then $\deg x\geq s\delta$, where $\delta=(\delta_1,\ldots, \delta_n)\in \N^n$ and $\delta_i=\min\{d_i\mid \deg f = (d_1,\ldots, d_n), 0\neq f\in I\}$.
\end{theorem}

\begin{proof}  Let $x\in \wsc{s}{I}$, $\deg x = (m_1,\ldots, m_n)$, and assume that $\deg x\not\geq  s\delta$.  Let $c\in R^\circ$ be homogeneous such that for all $q\gg0$, $cx^q\in I^{(s,q)}$.  For any such $q$, there exist homogeneous $y_q\in I^{\lceil sq\rceil}$ and $z_q\in I^{[q]}$, each of the same degree as $cx^q$, such that $cx^q=y_q+z_q$.  Since $y_q\in I^{\lceil sq\rceil}$,  if $y_q\neq 0$ then $\deg y_q\geq \delta\lceil sq\rceil$, and since $\deg x \not\geq s\delta$, for large enough $q$ we have that $\deg (cx^q)=q\cdot\deg x+\deg c \not\geq\delta\lceil sq\rceil $.  Therefore $y_q=0$ and $cx^q\in I^{[q]}$ for all $q\gg 0$, and hence $x\in I^*$.
\end{proof}

When the ideal we consider is primary to the homogeneous maximal ideal, we may conclude that all elements above a certain degree are in $\wsc{s}{I}$.

\begin{theorem} \label{thm - elements forced in} Let $k$ be a field of characteristic $p>0$, $(R,\m)$ an  $\N^n$-graded local finitely generated $k$-algebra, $I$ an $\m$-primary homogeneous ideal generated in degree at most $\delta$, and $s\geq 1$.  If $x\in R$ is a homogeneous element, $\deg x \neq (0,\ldots,0)$, and $\deg x \geq s\delta$, then $x\in \wsc{s}{I}$.
\end{theorem}

\begin{proof}  If $\Ht I=0$, then since $I$ is $\m$-primary, $R$ is a dimension 0 local ring.  In this case, all elements of $\m$ are nilpotent, and so since $\deg x\neq(0,\ldots,0)$, $x$ is nilpotent and hence in $\wsc{s}{I}$.  Therefore, we may assume that $\Ht I>0$.

 We reduce to the case that $R$ is $\N$-graded by ``flattening'' the grading on $R$.  Precisely, for $m\in \N$, let $R_m=\bigoplus_{|g|_1=m} R_g$, where the sum is taken over all degrees in the original grading whose sum of coordinates is equal to $m$.  Under this new grading, we still have that $\deg x\geq s\delta$, so we may assume $R$ is $\N$-graded.

Suppose that $\delta>0$.  Let $\Delta=\deg x\geq s\delta$, and let $f_1,\ldots, f_m$ be a set of homogeneous generators of $I$ with $\deg f_i\leq \delta$.  Since $I$ is $\m$-primary, we have that $k[f_1,\ldots, f_m]\subseteq R$ is integral, and so there exists an equation of integral dependence for $x^\delta$ over $k[f_1,\ldots, f_m]$:
 \begin{equation}\label{eqn of integral dependence}
\left(x^\delta\right)^N + a_1\left(x^\delta\right)^{N-1}+\cdots + a_{N-1}x^\delta +a_N=0.\end{equation}
we may choose the $a_i$ homogeneous, and so $\deg a_i = \Delta\delta i $ for each $i$.  Since each $a_i$ is a polynomial in the $f_i$, we have that $a_i\in I^{\Delta i}$ for all $i$.  Therefore $x^\delta\in \overline{I^{\Delta}}$, and so $x\in I_{\Delta/\delta}$. By Theorem~\ref{thm - rational powers to sclosure}, $x\in \wsc{\Delta/\delta}{I}\subseteq \wsc{s}{I}$.

Now suppose that $\delta = 0$, so that $I$ is generated by its degree 0 piece $I\cap R_0$.  In this case, for all $n\in \N$ we have that $I\cap R_n = (I\cap R_0)R_n$.  Since $I$ is $\m$-primary,
\[\infty > \lengthR{R}{R/I} = \sum_{n\in \N} \lengthR{R_0}{ \frac{R_n}{I\cap R_n}}
= \sum_{n\in \N} \lengthR{R_0}{\frac{ R_n}{(I\cap R_0)R_n}}
\]
Therefore, there exists $N\in\N$ such that if  $n\geq N$, $(I\cap R_0)R_n=R_n$, and by Nakayama's Lemma, $R_n=0$.  Hence any homogeneous $x\in R$ with $\deg x\geq 1$ is nilpotent, and so $x\in \wsc{s}{I}$. 
\end{proof}

For certain kinds of ideals in graded rings we can describe the $s$-closure completely in terms of rational powers.

\begin{theorem}\label{thm - sclosure is mixed power criterion}Let $R$ be a ring of characteristic $p>0$ which is $G$-graded for some semigroup $G$, $I\subseteq R$ a homogeneous ideal of positive height, and $s\geq 1$ rational.  If, for every $q\gg 0$ and $g\in G$, either $[I^{\lceil sq\rceil}]_g\subseteq [I^{[q]}]_g$ or $[I^{[q]}]_g\subseteq [I^{\lceil sq\rceil}]_g$, then $\wsc{s}{I}=I^*+I_s$.
\end{theorem}  

\begin{proof}  By Theorem~\ref{thm - rational powers to sclosure}, $I^*+I_s\subseteq \wsc{s}{I}$.  Let $x\in \wsc{s}{I}$ be homogeneous, and let $c\in R^\circ$ be homogeneous such that for all $q\gg 0$, $cx^q\in I^{(s,q)}=I^{\lceil sq\rceil}+I^{[q]}$.  For all sufficiently large $q$, since $cx^q$ is a homogeneous element, we have that $cx^q\in I^{\lceil sq\rceil}$ or $cx^q\in I^{[q]}$. 
If $x\notin I^*$, then for infinitely many $q$, $cx^q\notin I^{[q]}$.  Hence for infinitely many $q$, $cx^q\in I^{\lceil sq\rceil}$.  By Lemma~\ref{lem - rational power description}, $x\in I_s$.   Hence $x$ is in either $I^*$ or $I_s$, so $x\in I^*+I_s$.
\end{proof}

Situations where we might apply Theorem~\ref{thm - sclosure is mixed power criterion} include homogeneous ideals in rings with monomial-like gradings, in which  each graded piece is a 1-dimensional vector space over $k$. Examples of these include monomial ideals in polynomial rings and toric rings.

\section{When is The Weak $s$-closure a Closure?}\label{sec - ideal conditions}

In this section we consider a collection of conditions an ideal $I$ may have relating to its various $s$-closures.  In particular, we are concerned with when the $s$-closure is an honest closure itself, a property we refer to as $(ID_s)$ below.  Before defining the properties, we look at an example to show that, \emph{a priori}, there may be infinitely many distinct $s$-closures for a given ideal.  More precisely,
this example shows that it is possible for the quotient of two $s$-closures to have infinite length.  The example is based on \cite[Example 2.2]{Epstein-Yao-ExtensionsHK}

\begin{example}  \label{example - infinite length between s-closures} Let $R$ be a domain of characteristic $p>0$, $I\subseteq R$ an ideal, and $s>t\geq 1$ such that $\wsc{s}{I}\neq \wsc{t}{I}$.  Let $S=R[X]$, $J=IS$, and $z\in \wsc{t}{I}\setminus \wsc{s}{I}$.  We have that
\[\frac{\wsc{t}{J}}{\wsc{s}{J}}\supseteq \frac{\wsc{s}{J}+Sz}{\wsc{s}{J}}\cong \frac{Sz}{\wsc{s}{J}\cap Sz}\cong \frac{S}{\left(\wsc{s}{J}:_S z\right)}\]
We claim that $zX^n\notin \wsc{s}{J}$ for any $n\in \N$.  If there were such an $n$, then since $R[X]$ is naturally $\N$-graded, there would be an element $rX^m$ for some $0\neq r\in R$ and $m\in\N$ such that, for all sufficiently large powers $q$ of $p$,
\[rz^qX^{m+nq}=rX^m\cdot \left(zX^n\right)^q\in J^{(s,q)}=I^{(s,q)}S.\]
Therefore, for all large $q$, $rz^q\in I^{(s,q)}$, and so $z\in \wsc{s}{I}$, a contradiction.   Hence $X^n\notin \left(\wsc{s}{J}:_S z\right)$ for any $n$, and so
\[S\supseteq \left(\wsc{s}{J}:_S z\right)+(X)\supseteq \left(\wsc{s}{J}:_S z\right)+(X^2)\supseteq  \left(\wsc{s}{J}:_S z\right)+(X^3)\supseteq \cdots\]
is an infinite chain of descending ideals each of which contain $  \left(\wsc{s}{J}:_S z\right)$.  Therefore $\lambda\left(S/ \left(\wsc{s}{J}:_S z\right)\right)=\infty$, and so $\lambda\left(\wsc{t}{J}/\wsc{s}{J}\right)=\infty$.
\end{example}

\subsection*{Property $(ID_s)$: When Weak $s$-Closure is Equal to $s$-Closure}

As given in Definition \ref{def - s-closure}, the weak $s$-closure is not obviously a closure.   We now consider classes of ideals for which the two notions align.

\begin{definition}  Let $R$ be a ring of characteristic $p>0$ and $s\geq1$ a real number.  We say an ideal $I$ of $R$ has property $(ID_s)$ if the weak $s$-closure is idempotent on $I$, i.e.\ $\wsc{s}{I}=\scl{s}{I}$.  We say the ring $R$ has property $(ID_s)$ if every ideal of $R$ has property $(ID_s)$.
\end{definition}

Since weak $1$-closure is integral closure, any ring $R$ of positive characteristic has property $(ID_1)$.  If, further, $(R,\m)$ is local with infinite residue field, then $R$ has property $(ID_s)$ for any $s>\dim R$, since in this case weak $s$-closure is tight closure.

\subsection*{Property $(SM_s)$: When $s$-Closure is Characterized by $s$-Multiplicity}

In this section we restrict our attention to ideals of finite colength in local or graded local rings, homogeneous in the graded local case.  Membership in such ideals' tight or integral closure can be tested using the Hilbert-Kunz or Hilbert-Samuel multiplicity, under certain conditions.  The analogous  multiplicity function for $s$-closure is called $s$-multiplicity.

\begin{definition}\cite[Definition 1.3]{Taylor-Interpolating} Let $(R,\m)$ be a local (resp.\ graded local) ring of characteristic $p>0$ and dimension $d$, $I\subseteq R$ an $\m$-primary (homogeneous) ideal, and $s\geq 1$.  The \emph{$s$-multiplicity} of $I$ is 
\[e_s(I)=\lim_{q\to \infty}\frac{\lambda\left(R/I^{(s,q)}\right)}{q^d\mathcal{H}_s(d)},\]
where $\mathcal{H}_s(d) = \mathrm{vol} \{x\in [0,1]^d\mid |x|_1\leq s\}$.
\end{definition}

Originally, the $s$-multiplicity was defined only for local rings, but the graded local case is completely analogous.

 By \cite[Theorem 4.6]{Taylor-Interpolating}, if $x\in \wsc{s}{I}$, then $e_s(I+(x))=e_s(I)$.  When the converse also holds, the $s$-multiplicity becomes a powerful tool for studying the $s$-closure.
 
 \begin{definition}   Let $(R,\m)$ be a (graded) local ring of characteristic $p>0$ and $s\geq1$ a real number.  We say an $\m$-primary (homogeneous) ideal $I$ of $R$ has property $(SM_s)$ if one can test membership in the weak $s$-closure of $I$ using $s$-multiplicity, i.e.\ if $e_s(I+(x))=e_s(I)$ then $x\in \wsc{s}{I}$.
  We say the ring $R$ has property $(SM_s)$ if every $\m$-primary (homogeneous) ideal of $R$ has property $(SM_s)$.
\end{definition} 

Property $(SM_s)$ is stronger than property $(ID_s)$.

\begin{theorem}\label{theorem - SM to ID} Let $(R,\m)$ be a (graded) local ring of characteristic $p>0$ and $I$ a (homogeneous) $\m$-primary ideal of $R$.  If $I$ has property $(SM_s)$, then $I$ has property $(ID_s)$.
\end{theorem}

\begin{proof} By \cite[Theorem 4.6]{Taylor-Interpolating}, if $x\in \scl{s}{I}$, then $e_s(I+(x))=e_s(I)$, and since $I$ has property $(SM_s)$, we have that $x\in \wsc{s}{I}$.  Therefore $\wsc{s}{I}=\scl{s}{I}$.
\end{proof}

The following theorem shows that membership in the weak $s$-closure can be tested modulo minimal primes, which we will use to show that a large class of ideals has $(SM_s)$.  The proof of this result is based closely on the proof of part (e) of \cite[Proposition 10.1.2]{Bruns-Herzog-CMRings}.  In the following proof, the notation $\overline{x}$ always indicates the image of $x$ in the currently considered quotient ring

\begin{theorem}\label{thm - all ideals have MP} Let $R$ be a ring of characteristic $p>0$, $I\subseteq R$ an ideal, and $s\geq1$.  For any $x\in R$, $x\in \wsc{s}{I}$ if and only if $\overline{x}\in \wsc{s}{(IR/\p)}$ for all $\p\in \Min R$. 
\end{theorem}

\begin{proof}
If $x\in \wsc{s}{I}$, then there exists $c\in R^\circ$ such that for all $q\gg 0$, $cx^q\in I^{(s,q)}$, and therefore for any minimal prime $\p$, $\overline{c}\cdot\overline{x}^q=\overline{cx^q}\in I^{(s,q)}R/\p = (IR/\p)^{(s,q)}$.  Hence $\overline{x}\in\wsc{s}{(IR/\p)}$.

Let $\p_1,\p_2\ldots, \p_m$ be the minimal primes of $R$, suppose that for every $i$, $\overline{x}\in \wsc{s}{(IR/\p_i)}$, and choose $c_i\notin \p_i$ such that $\overline{c_ix^q}\in \left(IR/\p_i\right)^{(s,q)}$ for all $q\gg 0$.  Therefore $c_ix^q\in I^{(s,q)}+\p_i$ for all $i$ and all $q\gg 0$.  We may assume that $c_i\in R^\circ$; if not, by prime avoidance we may choose $c_i'$ such that $c_i'\in \p_j$ if and only if $c_i\notin \p_j$.  Therefore $c_i+c_i'\in R^\circ$, and furthermore, since $c_i'\in\p_i$, we have that $(c_i+c'_i)x^q\in I^{(s,q)}+\p_i$ for all $q\gg 0$.  Thus we may replace $c_i$ with $c_i+c'_i$.

For each $i$, let $d_i\in \left(\prod_{j\neq i}\p_j\right)\setminus \p_i$, and let $d=\sum_i c_id_i$.  We have that $d\in R^\circ$.  Now $\p_1\p_2\cdots \p_m\subseteq \sqrt{0}$, so let $q'$ be large enough that $(\p_1\p_2\cdots \p_m)^{[q']}=0$.  For all $q\gg 0$, and for any $i$, we have that
\[(c_id_i)^{q'}x^{qq'} = \left(c_ix^q\right)^{q'}d_i^{q'}
\in \left(I^{(s,q)}+\p_i\right)^{[q']}\left(\prod_{j\neq i} \p_i^{[q']}\right)
=\left(\left(I^{(s,q)}\right)^{[q']}+\p_i^{[q']}\right)\left(\prod_{j\neq i} \p_i^{[q']}\right)\subseteq I^{(s,qq')}.
\]
Therefore, $d^{q'}x^{qq'}\in I^{(s,qq')}$ for all $q\gg 0$, and so $x\in \wsc{s}{I}$.
\end{proof}

Theorem~\ref{thm - all ideals have MP} allows us to generalize \cite[Theorem 4.6]{Taylor-Interpolating} to the non-domain case.

\begin{theorem} \label{theorem - MP to SM}  Let $(R,\m)$ be a (graded) local ring of characteristic $p>0$ and $I$ a (homogeneous) $\m$-primary ideal.  If $R$ is $F$-finite, complete, and unmixed, then $I$ has property $(SM_s)$.
\end{theorem} 

\begin{proof}  Suppose that $x\in R$ such that $e_s(I+(x))=e_s(I)$. By the Associativity Formula for $s$-multiplicity, \cite[Corollary 3.10]{Taylor-Interpolating}, we have that
\[\sum_{\p\in \Assh R} e^{R/\p}_s((I+(x))R/\p)\lengthR{R_\p}{R_\p}=e_s(I+(x))=e_s(I)=\sum_{\p\in \Assh R} e^{R/\p}_s(IR/\p)\lengthR{R_\p}{R_\p}.\]
For each $\p\in\Assh R$,  $e_s^{R/\p}((I+(x))R/\p)\leq e_s^{R/\p}(IR/\p)$, and so we have equality for all{} such $\p$.  Since $R/\p$ is an $F$-finite complete domain for all $\p\in\Assh R$, $\overline{x}\in \wsc{s}{(IR/\p)}$ by \cite[Theorem 4.6]{Taylor-Interpolating}.   Since $R$ is unmixed, $\Assh R=\Min R$, and so by Theorem~\ref{thm - all ideals have MP}, we have that $x\in \wsc{s}{I}$.  Therefore, $I$ has property $(SM_s)$.
\end{proof}

\subsection*{Property $(LC_s)$: When Weak $s$-Closure is Left-Continuous}

Next we consider the condition that an $s$-closure is the intersection of all larger $s$-closures.  This property is enjoyed by rational powers, which are related to $s$-closures.

\begin{definition} Let $R$ be a ring of characteristic $p>0$ and $s>1$.  We say an ideal $I$ of $R$ has property $(LC_s)$ if the weak $s$-closure is left-continuous on $R$, i.e.\ $\wsc{s}{I}=\bigcap_{t<s}\wsc{t}{I}$. We say the ring $R$ has property $(LC_s)$ if every ideal of $R$ has property $(LC_s)$.
\end{definition}

The containment $\wsc{s}{I}\subseteq \bigcap_{t<s}\wsc{t}{I}$ always holds since the $s$-closure is monotonic in $s$.  In fact, we can say more, and the following Lemma will likely be important in later development of the theory of $s$-closures.

\begin{lemma}\label{lem - t-closure is s-closed} If $R$ is a ring of characteristic $p>0$, $I\subseteq R$ is an ideal, and $1\leq t<s$, then $\wsc{s}{\left(\wsc{t}{I}\right)}=\wsc{t}{I}$.
\end{lemma}

\begin{proof} Let $J=\wsc{t}{I}$.  We have that $J\subseteq \wsc{s}{J}$, since this holds for all ideals.  Now let $x\in \wsc{s}{J}$, let $c\in R^\circ$ such that $cx^q\in J^{(s,q)}$ for all $q\gg 0$, and let $d\in R^\circ$ such that $dJ^{[q]}\subseteq I^{(t,q)}$ for all $q \gg 0$.  Finally, let $q'$ be such that $q'(s-t)\geq \mu(J)$.  For $q\gg 0$, we have that

\[cd^{\lceil sq'-\mu(J)\rceil}x^{qq'} \in d^{\lceil sq'-\mu(J)\rceil}J^{(s,qq')} = d^{\lceil sq'-\mu(J)\rceil}J^{\lceil sqq'\rceil}+d^{\lceil sq'-\mu(J)\rceil}J^{[qq']}\]
Now $\lceil sq'-\mu(J)\rceil\geq q't\geq 1$, so $d^{\lceil sq'-\mu(J)\rceil}J^{[qq']}\subseteq I^{[qq']}\subseteq I^{(t,qq')}$.  Also, for $q\gg 0$ we have that
\[d^{\lceil sq'-\mu(J)\rceil}J^{\lceil sqq'\rceil}\subseteq d^{\lceil sq'-\mu(J)\rceil}\left(J^{[q]}\right)^{\lceil sq'-\mu(J)\rceil}
=\left(dJ^{[q]}\right)^{\lceil sq'-\mu(J)\rceil}\subseteq \left(I^{(t,q)}\right)^{\lceil sq'-\mu(J)\rceil} \subseteq \left(I^q\right)^{\lceil sq'-\mu(J)\rceil}.\]
Now $q\lceil sq'-\mu(I)\rceil \geq q(sq'-\mu(I)) \geq q(tq')$, and so
\[d^{\lceil sq'-\mu(J)\rceil}J^{\lceil sqq'\rceil}\subseteq \left(I^q\right)^{\lceil sq'-\mu(J)\rceil}\subseteq I^{\lceil tqq'\rceil}\subseteq I^{(t,qq')}.\]
Hence, we have that for all $q\gg 0$, $cd^{\lceil sq'-\mu(J)\rceil}x^{qq'}\in I^{(t,qq')}$, so $x\in \wsc{t}{I}=J$.  Thus $\wsc{s}{J}\subseteq J$.
\end{proof}

\begin{theorem} \label{theorem - LC to ID} If $R$ is a ring of characteristic $p>0$, $s>1$, and $I$ an ideal of $R$, then   $\scl{s}{I}\subseteq \bigcap_{t<s}\wsc{t}{I}$.  In particular, if $I$ has $(LC_s)$ then $I$ has $(ID_s)$.
\end{theorem}

\begin{proof} Let $J\subseteq \bigcap_{t<s}\wsc{t}{I}$ be any ideal.  For any $t<s$, we have that $J\subseteq \wsc{t}{I}$, and hence $\wsc{s}{J}\subseteq \wsc{s}{\left(\wsc{t}{I}\right)}=\wsc{t}{I}$ by Lemma~\ref{lem - t-closure is s-closed}.   Therefore $\wsc{s}{J}\subseteq \bigcap_{t<s}\wsc{t}{I}$.

Since $I\subseteq \bigcap_{t<s}\wsc{t}{I}$, and $\scl{s}{I}$ is obtained by applying the weak $s$-closure a finite number of times to $I$, we have that $\scl{s}{I}\subseteq \bigcap_{t<s}\wsc{t}{I}$.
\end{proof}

\subsection*{Property $(LS_s)$: When $s$-Closure is Left-Stable}

In this section we consider intervals of $s$ on which the weak $s$-closures of an ideal are constant.  Computations have shown that for many easily-understood ideals, the $s$-multiplicity is \emph{left-stable}, i.e.\ that for a given $s$ and ideal $I$, $\wsc{t}{I}=\wsc{s}{I}$ for all $t$ slightly smaller than $s$.  This is the strongest condition that we give a label to in this paper.  Before defining it, however, we give a weaker result that gives some insight into the intervals of $s$ on which $I$ has same $s$-closure.

\begin{theorem}\label{theorem - intervals of constancy}  Let $R$ be a ring of characteristic $p>0$ and $s\geq 1$.  There exists $\epsilon>0$ such that $\wsc{t}{I}=\wsc{s+\epsilon}{I}$ for any $t\in (s,s+\epsilon]$.
\end{theorem}

\begin{proof} Consider the chain of ideals 
\[\wsc{s+1}{I}\subseteq \wsc{s+1/2}{I}\subseteq \wsc{s+1/3}{I}\subseteq \cdots.\]
Since $R$ is noetherian, there exists $m\in \N$ such that for all $n\geq m$, $\wsc{s+1/n}{I}=\wsc{s+1/m}{I}$.  Hence, for any $t\in (s,s+1/m]$, there exists some $n$ such that $s+1/n<t$, and so $\wsc{s+1/m}{I}\subseteq \wsc{t}{I}\subseteq \wsc{s+1/n}{I}=\wsc{s+1/m}{I}$.
\end{proof}

Theorem~\ref{theorem - intervals of constancy} inspires a definition of jumping numbers for $s$-closure similar to that for test ideals.

\begin{definition} Let $R$ be a ring of characteristic $p>0$ and $I$ an ideal of $R$.  We say that a real number $s\geq 1$ is an \emph{$s$-jumping number} for $I$ if, for all $t>s$, $\wsc{s}{I}\neq \wsc{t}{I}$.
\end{definition}

Theorem~\ref{theorem - intervals of constancy} implies that jumping numbers cannot accumulate above a given $s\geq 1$.  That is, for any $s$ there is an $\epsilon>0$ such that there are no $s$-jumping numbers in $(s,s+\epsilon)$.  However, we do not have a theorem disproving the existence of such accumulations below $s$.  Therefore, we define a property based on that condition, which we show holds for some well-understood classes of ideals.

\begin{definition} Let $R$ be a ring of characteristic $p>0$ and $s>1$ a real number.  We say an ideal $I$ of $R$ has property $(LS_s)$ if the weak $s$-closure of $I$ is left-stable, i.e.\ there exists $\epsilon>0$ such that $\wsc{t}{I}=\wsc{s}{I}$ for all $s-\epsilon<t<s$.  We say the ring $R$ has property $(LS_s)$ if every ideal of $R$ has property $(LS_s)$. 
\end{definition}

Left-stability is a very strong condition, implying left-continuity and therefore idempotence.

\begin{theorem} \label{theorem - LS to LC} If $R$ is a ring of characteristic $p>0$, $s>1$, $I$ is an ideal of $R$, and $I$ has $(LS_s)$, then $I$ has $(LC_s)$.
\end{theorem}

\begin{proof} Since $I$ has $(LS_s)$, there exists $u<s$ such that $\wsc{u}{I}=\wsc{s}{I}$.  This implies that $\bigcap_{t<s}\wsc{t}{I}\subseteq \wsc{u}{I}=\wsc{s}{I}$.
\end{proof}

Left-stability is enjoyed by ideals that can be described by their rational powers.

\begin{theorem}\label{thm - sclosure=rational implies LS} If $R$ is a ring of characteristic $p>0$ and $I$ is an ideal of $R$ such that $\wsc{s}{I}=I^*+I_s$ for all $s\in\Q$, then $I$ has $(LS_s)$ for all $s> 1$.
\end{theorem} 

\begin{proof} By Theorem~\ref{thm - discrete rational powers}, there exists $e\in \N$ such that $I_\alpha = I_{\lceil \alpha e\rceil/e}$ for all $\alpha\in\Q_{>0}$.  Now let $t\geq 1$ be any real number such that $\frac{\lceil s e\rceil -1}{e}<t<s$.  Finally, let $\alpha,\beta\in\Q$ such that 
\[\frac{\lceil s e\rceil -1}{e}<\alpha\leq t<s\leq \beta \leq \frac{\lceil se\rceil}{e}.\]  We have that
\[\wsc{t}{I}\subseteq \wsc{\alpha}{I} = I^*+I_\alpha = I^*+I_\beta = \wsc{\beta}{I}\subseteq \wsc{s}{I}.\]
Therefore $\wsc{t}{I}=\wsc{s}{I}$.  Hence $I$ has $(LS_s)$.
\end{proof}

\begin{corollary}\label{cor - ideals with LS} The following classes of ideals have property $(LS_s)$ for all $s>1$.  All rings mentioned have characteristic $p>0$.
\begin{enumerate}
	\item \label{monom in poly} Monomial ideals in polynomial rings over a field.
	\item \label{monom in toric} Monomial ideals in affine semigroup rings over a field.
	\item \label{homog in monomial-like} Homogeneous ideals of positive height in graded rings in which each graded piece has length 1 over the zeroth piece. 
	\item \label{powers of maximal} Powers of $R_+$, where $R$ is an $\N$-graded ring, generated in degree $1$ over $R_0$, and $R_+$ is the ideal generated by all homogeneous elements of degree 1.
	\item \label{principal} Principal ideals.
\end{enumerate}
\end{corollary}

\begin{proof} 
Items \eqref{monom in poly} and \eqref{monom in toric} follow from part \eqref{homog in monomial-like} when we take the monomial $\N^n$-grading.  Thus, let $R=\bigoplus_{g\in G} R_g$ be a graded ring such that for all $g\in G$, $R_g$ has length $1$ over $R_0$.  Let $I\subseteq R$ be a homogeneous ideal, and fix $g\in G$. %Since $\lambda_{R_0}(R_g)=1$, we have that $R_g\cong R_0/\m$ as $R_0$-modules, where $\m$ is a maximal ideal of $R_0$.  
For any $q$, $[I^{[q]}]_g$ is an $R_0$-submodule of $R_g$, and therefore $[I^{[q]}]_g=R_g$ or $[I^{[q]}]_g=0$.  For any rational $s>1$, in the first case we have that $[I^{\lceil sq\rceil}]_g\subseteq [I^{[q]}]_g$ and in the second we have that $[I^{[q]}]_g\subseteq [I^{\lceil sq\rceil}]_g$.  Thus, by Theorem~\ref{thm - sclosure is mixed power criterion}, $\wsc{s}{I}=I^*+I_s$.  Hence by Theorem~\ref{thm - sclosure=rational implies LS}, $I$ has property $(LS_s)$ for all $s$. This proves \eqref{homog in monomial-like}.

For item \eqref{powers of maximal}, let $R$ be an $\N$-graded ring generated in degree 1 as an $R_0$-algebra.  Let $x_1,\ldots, x_t$ be a set of degree 1 generators for $R$ as an $R_0$-algebra, and let $R_+=(x_1,\ldots, x_t)$.  Finally, let $n\in\N$ and $I=(R_+)^n$.   Fix $g\in \N$, $q>0$, and $s\in \Q_{>0}$.  If $g\geq n\lceil sq\rceil$, then any $x\in R_g$ can be written as $x=rx_1^{a_1}\cdots x_t^{a_t}$, where $r\in R_0$ and $\sum_i a_i\geq n\lceil sq\rceil$.  Therefore, $x\in (R_+)^{n\lceil sq\rceil}=I^{\lceil sq\rceil}$.  Hence, if $g\geq n\lceil sq\rceil$, $[I^{\lceil sq\rceil}]_g=R_g$.
On the other hand, any homogeneous element of $I^{\lceil sq\rceil}$ must have degree at least $n\lceil sq\rceil$ since $I$ is generated in degree $n$.  Therefore, if $g<n\lceil sq\rceil$, then $[I^{\lceil sq\rceil}]_g=0$.  Hence for any $g$, $[I^{\lceil sq\rceil}]_g\subseteq [I^{[q]}]_g$ or $[I^{[q]}]_g\subseteq [I^{\lceil sq\rceil}]_g$, and by Theorem~\ref{thm - sclosure is mixed power criterion}, $\wsc{s}{I}=I^*+I_s$.  Hence by Theorem~\ref{thm - sclosure=rational implies LS}, $I$ has property $(LS_s)$ for all $s$.

Item \eqref{principal} follows from the fact that for a principal ideal $I$, $I^*=\overline{I}$, and so for all $s$, $\overline{I}=I^*\subseteq \wsc{s}{I}\subseteq \overline{I}$, hence $\wsc{s}{I}=\overline{I}$.\end{proof}

\subsection*{Relationships Between the Properties}
The various implications between the properties that we have defined can be summarized in the following figure.

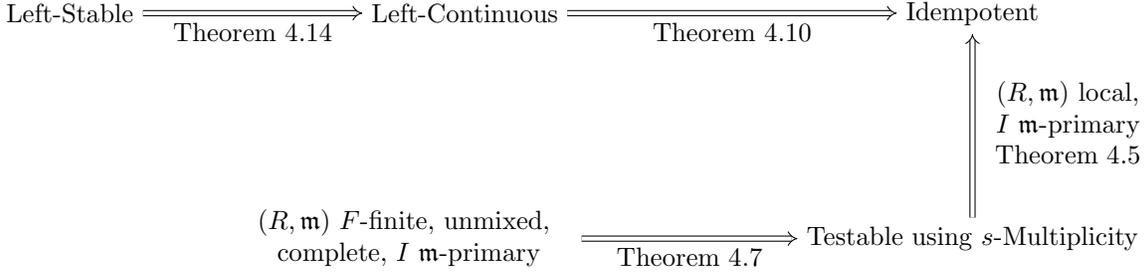
\begin{figure}[!htb]
\begin{tikzpicture}[scale=3]
\node (LS) at (0,1) {Left-Stable};
\node (LC) at (1.75,1) {Left-Continuous};
\node (ID) at (4,1) {Idempotent};
\node (MP) at (1.5,0) {$\begin{array}{c} (R,\m) \text{ $F$-finite, unmixed, }\\ \text{complete, } I \ \m\text{-primary}\end{array}$};
\node (SM) at (4,0) {Testable using $s$-Multiplicity};
\draw[double equal sign distance,-implies] (LS) -- (LC) node[midway,below]{Theorem \ref{theorem - LS to LC}};
\draw[double equal sign distance,-implies] (LC) -- (ID) node[midway,below]{Theorem \ref{theorem - LC to ID}};
\draw[double equal sign distance,-implies] (MP) -- (SM) node[midway,below]{Theorem \ref{theorem - MP to SM}} ;
\draw[double equal sign distance,-implies] (SM) -- (ID) node[midway,right] {$\begin{array}{c} (R,\m) \text{ local,}\\ I \ \m\text{-primary} \\\text{Theorem \ref{theorem - SM to ID}}\end{array}$} ;

\end{tikzpicture}
\end{figure}

\section{A Brian\c con-Skoda Theorem for $s$-Closure}\label{sec - BS}

The classical Brian\c con-Skoda Theorem describes containments between the integral closures of powers of an ideal and the powers themselves.  In particular, when $I$ is an ideal in a regular ring, we have that for all $n\in\N$, $\overline{I^{n+\mu(I)-1}}\subseteq I^n$. 
The statement is generalized by Hochster and Huneke in \cite[Theorem 5.4]{Hochster-Huneke-TightClosureBrianconSkoda}, who show that even in singular rings we have $\overline{I^{n+\mu(I)-1}}\subseteq (I^n)^*$ for all $n\in\N$.  This, combined with their proof that in regular rings all ideals are tightly closed, gives a new proof of the Brian\c con-Skoda Theorem.
In this section we develop a generalization of the Brian\c con-Skoda Theorem in positive characteristic.

\begin{theorem} \label{theorem - Briancon-Skoda} Let $R$ be a ring of characteristic $p>0$, $1\leq t< s$, and $I$ an ideal of $R$.  If $r\geq \frac{(\mu(I)-1)(s-t)}{t(s-1)}$, then for all $n\in\N$, $\wsc{t}{\left(I^{n+r}\right)}\subseteq \wsc{s}{\left(I^n\right)}$.
\end{theorem}

\begin{proof}  We consider two cases.  First, suppose that $n<\frac{\mu(I)-1}{s-1}$.  This implies that $r\geq \frac{(\mu(I)-1)(s-t)}{t(s-1)}>\frac{n(s-t)}{t}$.  If $q$ is large enough that $\frac{n(s-t)}{t}+\frac{n}{tq}\leq r$, then
\[(n+r)\lceil tq\rceil \geq ntq+rtq\geq ntq+n(s-t)q+n=nsq+n\geq n\lceil sq\rceil.\]
Therefore, for all $q\gg 0$,
\[\left(I^{n+r}\right)^{(t,q)}=I^{(n+r)\lceil tq\rceil}+\left(I^{n+r}\right)^{[q]}\subseteq I^{n\lceil sq\rceil}+\left(I^{n}\right)^{[q]}=\left(I^{n}\right)^{(s,q)}.
\]
Therefore $\wsc{t}{\left(I^{n+r}\right)}\subseteq \wsc{s}{\left(I^n\right)}$.

Now suppose that $n\geq \frac{\mu(I)-1}{s-1}$.  In this case we have that
\[(n+r)t=n+n(t-1)+rt \geq n+\frac{(\mu(I)-1)(t-1)}{s-1}+\frac{(\mu(I)-1)(s-t)}{s-1}=n+\mu(I)-1\]
and hence, for any $q$,
\[\left(I^{n+r}\right)^{(t,q)}=I^{(n+r)\lceil tq\rceil}+\left(I^{n+r}\right)^{[q]}\subseteq I^{(n+\mu(I)-1)q}+\left(I^{n}\right)^{[q]}\subseteq \left(I^{n}\right)^{[q]}.\]
Therefore, $\wsc{t}{\left(I^{n+r}\right)}\subseteq\left(I^n\right)^*\subseteq \wsc{s}{\left(I^n\right)}$.
\end{proof}

Theorem~\ref{theorem - Briancon-Skoda} recovers the classical Brian\c con-Skoda Theorem by taking $t=1$ and $r=\mu(I)-1$.  In particular, we note that Theroem~\ref{theorem - Briancon-Skoda} does not give us a stronger version of the theorem in the case that one of our closures is integral closure. 

We record two immediate corollaries, one of the statement of Theorem~\ref{theorem - Briancon-Skoda}  and one of its proof.

\begin{corollary}  Suppose $(R,\m)$ is a local ring with dimension $d$, characteristic $p>0$, and infinite residue field, let $I\subseteq R$  be an ideal with reduction number $w$, and let $1\leq t< s$.  If $r\geq \frac{(d-1)(s-t)}{t(s-1)}$, then for all $n\in\N$, $\wsc{t}{\left(I^{n+r+w}\right)}\subseteq \wsc{s}{\left(I^n\right)}$.
\end{corollary}

\begin{proof}  Since $R$ has infinite residue field, $I$ has a minimal reduction $J$ with reduction number $w$ and generated by $d$ elements.  Therefore, 
\[\wsc{s}{\left(I^{n+r+w}\right)}\subseteq \wsc{t}{\left(J^{n+r}\right)}\subseteq \wsc{s}{\left(J^n\right)}\subseteq \wsc{s}{\left(I^n\right)}.\qedhere\]
\end{proof}

\begin{corollary} \label{cor - asymptotic gives tight} Let $R$ be a ring of characteristic $p>0$, $1\leq s$, and $I$ an ideal of $R$.  If $n\in\N$ and $r\geq \frac{1}{s}\left(\mu(I)-1-n(s-1)\right)$, then $\wsc{s}{\left(I^{n+r}\right)}\subseteq \left( I^n\right)^*$.  In particular, if $n\geq \frac{\mu(I)-1}{s-1}$, then $\wsc{s}{\left(I^n\right)}=\left(I^n\right)^*$.
\end{corollary}

One way of interpreting Corollary~\ref{cor - asymptotic gives tight} is that asymptotically, as we take powers of an ideal, each $s$-closure with $s>1$ eventually collapses and becomes tight closure. In general, for smaller $s$, we must take ever higher powers of $I$ to realize this collapse; i.e., there is in general no uniform power beyond which every $s$-closure for $s>1$ is tight closure, as the following example shows.

\begin{example}
Let $I=(x^3,y^3)\subseteq k[x,y]$, where $k$ is a field of characteristic $p>0$.  By Theorem~\ref{thm - sclosure is mixed power criterion}, for any rational $s$, $\wsc{s}{(I^n)}=(I^n)^*+(I^n)_s=I^n+I_{ns}$.  Now $I_{ns}$ is generated by all monomials with degree at least $3ns$.  Thus, if $1<s<1+\frac{1}{3n}$, we have that $\deg(x^{3n-1}y^2)=3n+1=3n\left(1+\frac{1}{3n}\right)\geq 3ns$.  Therefore $x^{3n-1}y^2\in \wsc{s}{(I^n)}$, but $x^{3n-1}y^2\notin I^n = (I^n)^*$.
\end{example}

\bibliographystyle{alpha}
\bibliography{wdtbiblio}
\end{document}